\documentclass[12pt]{article}
\usepackage{amsmath}
\usepackage{hyperref}
\usepackage{amsthm}
\usepackage{amssymb}
\usepackage{mathrsfs}
\usepackage{color}
\usepackage{multicol}
\usepackage{relsize}
\usepackage{lscape}
\usepackage{pdfsync}
\mathchardef\mhyphen="2D

\newcommand{\beq}{\begin{eqnarray*}}
\newcommand{\feq}{\end{eqnarray*}}
\newcommand{\beqn}{\begin{eqnarray}}
\newcommand{\feqn}{\end{eqnarray}}
\newcommand{\witi}{\widetilde}
\newcommand{\veps}{\varepsilon}

\newcommand{\ints}{\mathbb{Z}}
\newcommand{\reals}{\mathbb{R}}

\newcommand{\nats}{\mathbb{N}}

\newcommand{\bbN}{\mathbb{N}}

\newtheorem{theorem}{Theorem}
\makeatletter \@addtoreset{theorem}{section}\makeatother

\newtheorem{definition}[theorem]{Definition}
\newtheorem{lemma}[theorem]{Lemma}
\newtheorem*{theorem*}{Theorem}

\newtheorem{proposition}[theorem]{Proposition}
\newtheorem{corollary}[theorem]{Corollary}

\newtheorem{remark}[theorem]{Remark}

\newcommand{\Var}{\mathrm{Var}}
\newcommand{\ra}{\rightarrow}

\title{On the range of the transient frog model on $\ints$}
\author{
Arka Ghosh\footnote{Department of Statistics, Iowa State University, Ames, IA 50011, USA; e-mail: apghosh@iastate.edu}
\and
Steven Noren\footnote{Department of Mathematics, Iowa State University, Ames, IA 50011, USA; e-mail: srnoren@iastate.edu}
\and
Alexander Roitershtein\footnote{Department of Mathematics, Iowa State University, Ames, IA 50011, USA; e-mail: roiterst@iastate.edu} }
\begin{document}
\maketitle
\begin{abstract}
In this paper, we observe the frog model, an infinite system of interacting random walks, on $\ints$ with an asymmetric underlying random walk. Under the assumption of transience with a fixed frog distribution, we construct an explicit formula for the moments of the lower bound of the model's long-run range, as well as their asymptotic limits as the drift of the underlying random walk vanishes. We also provide conditions in which the lower bound can be scaled to converge in probability to the degenerate distribution at $1$ as the drift vanishes.
\end{abstract}
{\em MSC2010: } Primary: 60J10, 60K35; Secondary:  60E05, 33D05.\\
\noindent {\em Keywords:} Frog model, interacting random walks, random walks range, $q$-series.
\section{Introduction}
\label{intro}
Consider the following interacting random walks model on $\ints$: initially at each site $x$ there is a fixed number $\eta_x$ of sleeping particles (``frogs"), and there is a certain number ($\eta_0$) of active frogs at the origin. The active frogs perform in discrete time, simultaneously and independently of each other, a biased (say, to the right) nearest-neighbor random walk on $\ints$. When an active frog visits a site $x$, it activates the $\eta_x$ sleeping frogs at $x$, in which each active frog performs the same underlying random walk starting from its initial location, all random walk transitions being independent of each other. The active frogs continue to visit other sleeping frogs and activate them. This model for an infinite number of interacting random walkers is called the frog model on $\ints$ (with drift).
\par
Following \cite{gantert}, we define the notions of recurrence and transience for the frog model as follows:
\begin{definition}
\label{rt}
\normalfont
A frog model is called \emph{recurrent} if $0$ is visited infinitely often by active frogs w.p.1, and \emph{transient} if $0$ is visited only finitely often w.p.1.
\end{definition}
It is shown in \cite{gantert} that the zero-one dichotomy actually takes place, namely a one-dimensional frog model is either transient or recurrent. Both necessary and sufficient conditions for recurrence of the frog model on $\ints$ based on the configuration of frogs and the drift of the random walk are provided in \cite{gantert}. Recurrence for variants of the frog model on more general graphs have been first explored in \cite{telcs} (for the symmetric random walk on $\ints^d$) and subsequently in \cite{popov02}, \cite{popov01} and \cite{popov03}. Shape theorems for the model in $\ints^d$ have been obtained in \cite{shape,shape1}. For further background on the frog model and its variants, refer to \cite{popov03}. For an account of the most recent activity in the area see \cite{LG,zidi,trees,treesp}. In particular, \cite{zidi} generalizes a recurrence criterion of \cite{gantert} to a model in $\ints^d,$ \cite{trees} and \cite{treesp} provide recurrence and transience criteria for the frog model on trees, and \cite{LG} studies survival of particles in a one-dimensional variation of the model, also partially extending some of the results of \cite{gantert}.
\par
The frog model can be interpreted as an information spreading network \cite{shape1,popov03}. The underlying idea is that an active frog holds some information and shares it with sleeping frogs when they meet, activating the sleeping frogs who then spread the information along their random walk path. A closely related to our model
\emph{particle process on $\ints,$} describing the evolution of a virus in an infinite population (e.\,g., computer network),
has been considered in \cite{virus,virus1} and \cite{LG}. The model is also a discrete-time relative of the one-dimensional \emph{stochastic combustion process} studied in \cite{combustion, combustion1}.

\par
In this article, we will explore the behavior of the frog model, in particular its range, when transience is assumed. We specifically introduce a drift component to the random walk and explore how its magnitude effects the range of visited sites in the model. Each active frog will move one integer to the right with probability $p\in(\frac12,1)$, or one integer to the left with probability $1-p$.
Thus the underlying random walk is transient to the right. We define the drift constant \beqn
\label{rho}
\rho:=\frac{1-p}{p}\in(0,1).
\feqn
The drift term $\rho$ can be seen as a measure of ``transience" of this frog model; small values of $\rho$ indicate more frequent rightward movement by the frogs, whereas values of $\rho$ close to 1 more closely resemble recurrence with a slight rightward drift (see, for instance, formula \eqref{returns} below for a concrete
random walk result). Of particular interest is the collective behavior of the frogs as $\rho\uparrow1$.
\par
By Theorem~2.1 in \cite{gantert}, this frog model is transient when we assume an identical distribution of frogs on the nonnegative sites. It is easy to see that w.p.1 there must be only a finite number of visited sites to the left of the origin. That is, in the language of \cite{LG}, transience implies \emph{local extinction} for our model.
\par
In this article, while assuming $\eta_x=0$ for all $x<0$, we will first explore the single-frog case, i.e., the frog model in which $\eta_x=1$ for all $x\geq0$ (cf. \cite{LG,virus,virus1}). We will provide exact and asymptotic results for the moments of the lower bound of the range, which will be used in convergence theorems. After that, we will move to more general choices for $\eta$ and show that, under certain conditions, the frog model's lower bound will behave asymptotically similar to that of the single-frog case. Finally, we will provide asymptotic bounds for moments of the frog model range when $\eta$ is supported on all of $\ints$.
\par
The rest of the paper is organized as follows. A short Section~\ref{prelim} introduces notations and certain technical tools necessary for our proofs.
The three subsequent sections constitute the main body of the article.
The single-frog case is considered in Section~\ref{single}. A class of more general initial configurations of frogs $\eta$ is discussed in
Section~\ref{mgeneral}. The consideration of configurations supported on $\ints$ is discussed in Section~\ref{zconf}.
\section{Preliminaries}
\label{prelim}
For our calculations of the density and moments of the random variable representing the range of the frog model,
we make use of common notation in analytic number theory and combinatorics:
\begin{definition}
\label{qp}
\normalfont
For all $a,q$ such that $|q|<1$, the \emph{$q$-Pochhammer symbol} is defined by
\[
(a;q)_c:=\prod_{j=0}^{c-1}(1-aq^j)\hspace{0.5cm}\text{for } c\in\nats \cup \{\infty\}.
\]
\end{definition}

The $q$-Pochhammer symbol is one of the key functions in the construction of $q$-analogs in number theory, and is often used in the theory of basic hypergeometric functions and analytic combinatorics \cite{qseries,qseries1,qstuff}. For example, \cite{qstuff} provides the following pair of power series identities:
\begin{eqnarray} \label{powerseries1}
(z;q)_\infty=\sum_{n=0}^\infty \frac{(-1)^n q^{n(n-1)/2}}{(q;q)_n} x^n,
\end{eqnarray}
\begin{eqnarray}
\label{powerseries2}
\frac1{(z;q)_\infty}=\sum_{n=0}^\infty \frac{z^n}{(q;q)_n}.
\end{eqnarray}
The power series above reveal that, for fixed nonzero $q\in(-1,1)$, $g(z):=(z;q)_\infty$ and $h(z):=\frac1{(z;q)_\infty}$ are both analytic functions on $(0,1)$. Hence, for any analytic function $f:(0,\infty)\ra(0,1)$, the compositions $g\circ f$ and $h \circ f$ are both analytic functions, and thus infinitely differentiable.
\par
For the calculation of generating functions, it is also useful to recall the following $q$-Pochhammer equality:
\begin{eqnarray} \label{qpoch}
(q^{x+1};q)_\infty=\frac{(q;q)_\infty}{(q;q)_{x}},\hspace{0.5cm} x\in \nats
\end{eqnarray}
\begin{definition}
\label{gf}
\normalfont
Let $q\in (-1,1)$. The \emph{$q$-gamma function} $\Gamma_q$ is defined as
\[
\Gamma_q (z)=\frac{(q;q)_\infty}{(q^z;q)_\infty}(1-q)^{1-z}
\]
The \emph{$q$-digamma function} $\psi_\rho$ is defined as
\begin{eqnarray}
\label{qpolygamma}
\psi_q(z)=\frac1{\Gamma_q(z)} \frac{\partial \Gamma_q(z)}{\partial z} = -\ln(1-q)+\ln q\sum_{n=0}^\infty \frac{q^{n+z}}{1-q^{n+z}}
\end{eqnarray}
\end{definition}
To facilitate our calculation of the range moments, we need to develop notation for Bell polynomials. The Bell polynomials are defined (see, for instance, \cite{ecomb, cumulants}) as the triangular array of polynomials $B_{m,k}$, $m\geq k$, given by
\begin{eqnarray}
\label{bellpoly}
&&
\nonumber
B_{m,k}(x_1,x_2,\dots,x_{n-k+1})
\\
&&
\qquad
=\sum\frac{m!}{k_1!k_2!\cdots k_{m-k+1}!}\left(\frac{x_1}{1!}\right)^{k_1}\left(\frac{x_2}{2!}\right)^{k_2}\cdots\left(\frac{x_{m-k+1}}{(m-k+1)!}\right)^{k_{m-k+1}},
\end{eqnarray}
where the sum is taken over all sequences $\{k_1,k_2,\dots,k_{n-k+1}\}\subset\nats\cup\{0\}$ satisfying
\[k_1+k_2+\cdots+k_{m-k+1}=k \quad \mbox{\rm and} \quad k_1+2k_2+\cdots+(m-k+1)k_{m-k+1}=m.\]
We define the \emph{$m^\text{th}$ complete Bell polynomial} to be
\[
B_m(x_1,x_2,\dots,x_m)=\sum_{k=1}^m B_{m,k}(x_1,x_2,\dots,x_{m-k+1}).
\]

It will be important for our asymptotic analysis of the range of the model to note that $B_m(x_1,x_2,\dots,x_m)$ has only one term of the form $x_1^j$, namely $x_1^m$, and that the coefficient for this term in $B_m(x_1,x_2,\dots,x_m)$ is 1.
\par
While the Bell polynomials has many intriguing details that can be explored in combinatorial number theory (see, for instance, \cite{ecomb} and references therein) , we will mostly concern ourselves with their presence in the celebrated Fa\`{a} di Bruno's formula \cite{dibruno, dibruno1} for higher derivatives of composite functions:
\begin{eqnarray} \label{bruno}
\left(\frac{d}{dt}\right)^m f\bigl(g(t)\bigr)=\sum^m_{k=1}f^{(k)}\bigl(g(t)\bigr)~B_{m,k}\bigl(g'(t),g''(t),\dots,g^{(m-k+1)}(t)\bigr).
\end{eqnarray}
Throughout this paper, we make use of the common asymptotic notation: $f(x)\sim g(x)$ as $x\ra c$ if and only if $\lim_{x\ra c}\frac{f(x)}{g(x)}=1$.
\section{The single frog per site case}
\label{single}
We will first assume that $\eta_x=1$ for each nonnegative integer $x$ and $\eta_x=0$ elsewhere. Let $\rho\in(0,1)$ and consider the corresponding frog model. Let $W_\rho$ represent the random variable for the negative of the minimum of the visited sites in this model. For convenience, we will construct a family of mutually independent random variables $(X_\rho)_{\rho\in(0,1)}$ that all share the same probability space $(\Omega,\mathcal{F},P)$ such that $X_\rho$ and $W_\rho$ share the same distribution.
\par
The distribution function of $X_\rho$ can be easily found by observing that, by the rightward tendency of the initial active frog, all frogs w.p.1 will eventually be woken. Furthermore, if $(S_n)_{n\geq 0}$ has the distribution of the underlying random walk, then
\beqn
\label{returns}
P(S_n=0~\mbox{for some}~n\geq 1\,|\,S_0=1)=\rho,
\feqn
where $\rho$ is introduced in \eqref{rho}. With this observation, we see that for all $x\geq 0$,
\begin{eqnarray}
P(X_\rho\leq x)=P(X_\rho < x+1)=\prod_{j=1}^\infty (1-\rho^{x+j})=(\rho^{x+1};\rho)_\infty.
\end{eqnarray}
It is a simple exercise to see that $P(X_\rho=0)=\prod_{j=1}^\infty (1-\rho^j)=(\rho;\rho)_\infty$. For all $x>0$, the value of $X_\rho$'s probability density function is
\begin{eqnarray}
\begin{split}P(X_\rho=x)&=P(X_\rho\leq x)-P(X_\rho\leq x-1)=\prod_{j=1}^\infty (1-\rho^{x+j})-\prod_{j=0}^\infty (1-\rho^{x+j})\\&=(1-(1-\rho^x))\prod_{j=1}^\infty (1-\rho^{x+j})=\rho^x(\rho^{x+1};\rho)_\infty.
\end{split}
\end{eqnarray}
With the density known, we would now like to study the behavior of $X_\rho$ for values of $\rho$ close to 1, where the frog model more closely resembles the recurrent case. Objects to observe the concentration of the distribution of $X_\rho$ include the central statistics of the random variable, such as the mode and the expectation. For fixed $\rho$, we define a unique representative of the mode statistic, named $M_\rho$, by
\[
M_\rho:=\min\bigl\{x\geq0 : P(X_\rho=x)\geq P(X_\rho=n)\text{ for all } n\geq0\bigr\}.
\]
While the mode statistic is not usually observed compared to other central statistics of a class of random variables, it is a quick calculation that can often provide insight on asymptotic. Also, as we observe later, the concentration of $X_\rho$ around its mode will be influential to its limiting behavior. With that in mind, we present the following result:

\begin{theorem}
\label{mode}
Let $\rho\in(0,1)$, and let $M_\rho$ be the mode statistic for the distribution of $X_\rho$. Then,
\[
\Bigl\lfloor\dfrac{\ln(1-\rho)-\ln\rho}{\ln\rho}\Bigr\rfloor\leq M_\rho \leq \Bigl\lfloor\dfrac{\ln(1-\rho)-\ln(2-\rho)}{\ln\rho}\Bigr\rfloor,
\]
where $\lfloor a \rfloor$ is the largest integer less than or equal to $a\in\reals$. In particular, $M_\rho \sim \frac{\ln(1-\rho)}{\ln\rho}$ as $\rho\ra 1$.
\end{theorem}
Note that for all $\rho\in(0,1)$, the difference between the two bounds in the theorem's conclusion always belongs to the open interval $(0,2).$ Hence, even for values of $\rho$ that make the bounds extremely large, the theorem narrows down $M_\rho$ to two possibilities.
\begin{proof}[Proof of Theorem~\ref{mode}]
Considering $x$ to be a continuous variable, we note that
\begin{eqnarray}\label{derivative}
\frac{d}{dx} (\rho^{x+1};\rho)_\infty = \frac{d}{dx} \exp\Bigl(\sum_{j=1}^\infty \ln(1-\rho^{x+j})\Bigr)=\Bigl(\frac{d}{dx}\sum_{j=1}^\infty \ln(1-\rho^{x+j})\Bigr)\,(\rho^{x+1};\rho)_\infty,
\end{eqnarray}
assuming that the derivative of the series exists. By Theorem 7.17 in \cite{rudin}, the existence of the derivative depends on the uniform convergence of the series of derivatives on closed intervals. It is enough to confirm that $\sum_{j=1}^\infty ||\frac{\rho^{x+j}}{1-\rho^{x+j}}||_\infty$ absolutely converges for any closed interval $[a,b]$ we choose, where $||f||_\infty=\sup\{|f(x)|:x\in[a,b]\}$. The series is bounded above by $\frac{\rho^x}{(1-\rho^x)^2}$, so it converges absolutely. Hence, the series in \eqref{derivative} is differentiable, and furthermore, by \cite[Theorem 7.17]{rudin}, the derivative of the series is the series of the derivatives. Thus,
\[
\frac{d}{dx} (\rho^{x+1};\rho)_\infty = -\ln(\rho)(\rho^{x+1};\rho)_\infty \sum_{j=1}^\infty \frac{\rho^{x+j}}{1-\rho^{x+j}}.
\]
Using the above derivative, we can find the derivative of the density function for $X_\rho$:
\[\begin{split}
\frac{d}{dx}\Bigl\{(\rho^{x+1};\rho)_\infty \rho^x\Bigr\}&=\frac{d}{dx}\Bigl\{\prod_{j=1}^{\infty}(1-\rho^{x+j}) \rho^x\Bigr\}=\ln(\rho)\rho^x(\rho^{x+1},\rho)_\infty\Bigl(1-\sum_{j=1}^\infty \frac{\rho^{x+j}}{1-\rho^{x+j}}\Bigr).
\end{split}
\]
For the given $\rho$, $\ln(\rho)\rho^x(\rho^{x+1},\rho)_\infty$ is nonzero for all $x>0$. Therefore, by the $1^{\text{st}}$-derivative test for critical points, $M_\rho$ is an integer within 1 away from the positive value $m$ such that
\[
\sum_{j=1}^\infty \frac{\rho^{m+j}}{1-\rho^{m+j}}=1.
\]
Note that $\rho^{x+j}\leq\frac{\rho^{x+j}}{1-\rho^{x+j}}\leq \frac{\rho^{x+j}}{1-\rho^{x+1}}$ for all $x>0$ and all $j\in\nats$. Thus, using the formula for the geometric sum,
\[
\frac{\rho^{m+1}}{1-\rho}=\sum_{j=1}^\infty \rho^{m+j} \leq \sum_{j=1}^\infty \frac{\rho^{m+j}}{1-\rho^{m+j}}=1 \leq \sum_{j=1}^\infty \frac{\rho^{m+j}}{1-\rho^{m+1}} = \frac1{1-\rho^{m+1}}\frac{\rho^{m+1}}{1-\rho}.
\]
The result follows accounting for the fact that $M_\rho$ is integer-valued.
\end{proof}

We would now like to compare the mode of $X_\rho$ with the moments of the random variable. Calculating the moments directly from definition can be quite a challenge, given the convoluted expression of the density in \eqref{density}. Therefore, we take the more circuitous option of first finding the cumulants of $X_\rho$.
\par
We first define the cumulant generating function of $X_\rho$. With the moment generating function $M_{X_\rho}(t):=E(e^{tX_\rho})$, the cumulant generating function $g_\rho(t)$ of $X_\rho$ is defined as the natural logarithm of $M_{X_\rho}(t)$:
\[
g_\rho(t):=\log\bigl(M_{X_\rho}(t)\bigr)=\log\bigl(E(e^{tX_\rho})\bigr).
\]
We then define the \emph{$m^{\text{th}}$ cumulant} $\kappa_\rho^{(m)}$ of $X_\rho$ to be the $m^{\text{th}}$ derivative of the cumulant generating function evaluated at 0:
\[
\kappa_\rho^{(m)}:=g_\rho^{(m)}(0).
\]
Cumulants are known in probability to be an alternative to moments which are difficult to calculate explicitly. Cumulants can also be used to determine moments through the use of Fa\`{a} di Bruno's formula \eqref{bruno} (see, for instance, \cite{dibruno1, cumulants} and references therein), which is the direction we will take. It turns out that the cumulants of $X_\rho$, though unable to be written down using fundamental functions, can be expressed as straight-forward series representations.
\begin{lemma}\label{cumulants}
For each $\rho\in(0,1)$, the cumulant generating function of $X_\rho$ is
\begin{eqnarray}\label{cgf}
g_\rho(t)=\sum_{k=1}^\infty \ln\Bigl(\frac{1-\rho^k}{1-e^t\rho^k}\Bigr).
\end{eqnarray}
Furthermore, for each $m\in\nats$, the $m^{\text{th}}$-cumulant of $X_\rho$ is
\begin{eqnarray}\label{mcumulants}
\kappa_\rho^{(m)}=\sum_{k=1}^\infty \frac{k^{m-1}\rho^k}{1-\rho^k}.
\end{eqnarray}
\end{lemma}
\begin{proof}
Using \eqref{powerseries2}, we calculate the moment generating function of $X_\rho$:
\begin{eqnarray}
\label{mgf}
M_{X_\rho}(t):=E(e^{tX_\rho})=(\rho;\rho)_\infty \sum_{x=0}^\infty \frac{e^{tx}\rho^x}{(\rho;\rho)_x}=\frac{(\rho;\rho)_\infty}{(e^t\rho;\rho)_\infty}= \prod_{k=1}^\infty \frac{1-\rho^k}{1-e^t\rho^k}.
\end{eqnarray}
Taking the natural logarithm of \eqref{mgf} gives us \eqref{cgf}. To find $\kappa_\rho^{(m)}$, we first find the $m^{\text{th}}$ derivative of $g_\rho(t)$:
\[
\begin{split}
g^{(m)}_\rho(t)&= \sum_{k=1}^\infty \Bigl(\frac{d}{dt}\Bigr)^m \ln\Bigl(\frac{1-\rho^k}{1-e^t\rho^k}\Bigr) = \sum_{k=1}^\infty \Bigl(\frac{d}{dt}\Bigr)^m \Bigl(\ln(1-\rho^k)+\sum_{j=1}^\infty \frac{(e^t\rho^k)^j}{j}\Bigr)\\
&= \sum_{k=1}^\infty \sum_{j=1}^\infty j^{m-1}(e^t\rho^k)^j=\sum_{j=1}^\infty \sum_{k=1}^\infty j^{m-1}(e^t\rho^k)^j=\sum_{j=1}^\infty \frac{j^{m-1}e^{tj}\rho^j}{1-\rho^j}.
\end{split}
\]
Note that moving the derivative inside of the summation is justified by Theorem 7.17 in \cite{rudin}, at least for $t$ in a closed interval $[-\delta,\delta]$ for small enough $\delta>0$. Setting $t=0$ gives us \eqref{mcumulants}.
\end{proof}
Finding the moments of a random variable through its cumulants is a well-known technique, but we show the details here for completeness. Note that $M_{X_\rho}(t)=e^{g_\rho(t)}$ is the exponential function composed with another function. Applying Fa\`{a} di Bruno's formula \eqref{bruno} to $M_{X_\rho}(t)$ shows that
\beq
M^{(m)}_{X_\rho}(t)&=&e^{g_\rho(t)} \sum^m_{k=1} B_{m,k}(g_\rho'(t),g_\rho''(t),\dots,g_\rho^{(m-k+1)}(t))
\\
&=&
M_{X_\rho}(t) B_m(g_\rho'(t),g_\rho''(t),\dots,g_\rho^{(m)}(t)).
\feq
Setting $t=0$, we arrive at the following result.
\begin{theorem}
\label{moments}
Using the notation from Section~\ref{intro} and Lemma \ref{cumulants}, for each $\rho\in(0,1)$ and for all $m\in \nats$,
\begin{eqnarray}\label{mmoments}
E(X_\rho^m)=B_m\left(\kappa_\rho^{(1)},\kappa_\rho^{(2)},\cdots,\kappa_\rho^{(m)}\right).
\end{eqnarray}
\end{theorem}


The exact calculation of the moments in Theorem~\ref{moments} can be quite unwieldy for large values of $m$. Thus, it's often more insightful to observe simpler asymptotic formulas for the moments instead. In \cite{pippenger}, the following asymptotics as $\rho\uparrow1$ were proven for the input series of the Bell polynomials in \ref{moments}:
\begin{eqnarray}\label{asymp1}
\sum_{k=1}^\infty \frac{\rho^k}{1-\rho^k} \sim \frac1{1-\rho}\ln\frac1{1-\rho} \sim \frac{\ln(1-\rho)}{\ln\rho}
\end{eqnarray}
and, for $j\geq 1$,
\begin{eqnarray}\label{asymp2}
\sum_{k=1}^\infty \frac{k^j\rho^k}{1-\rho^k} \sim \frac{j!\zeta(j+1)}{(1-\rho)^{j+1}} \sim \frac{j!\zeta(j+1)}{-\ln^{j+1}\rho},
\end{eqnarray}
where $\zeta(j)=\sum_{k=1}^\infty 1/k^j$ is the Riemann zeta function.
\par
As noted in the previous discussion on Bell polynomials, for each fixed $m\in\nats$, the polynomial $B_m(x_1,\dots,x_m)$ includes only one term of the form $x_1^m$, and this term has coefficient 1. Recalling the construction of the Bell polynomials \eqref{bellpoly} and the pigeonhole principle, it is clear all other terms of $B_m(x_1,\dots,x_m)$ is a multiple of some variable other than $x_1$.
\par
Note that from \eqref{asymp1}, $\bigl(\frac{\ln\rho}{\ln(1-\rho)}\bigr)^m \bigl(\kappa_\rho^{(1)}\bigr)^m\ra 1$ as $\rho\uparrow1$.
Also, it follows from \eqref{asymp2}, that for all $j\in\{2,3,4,\dots,m\}$,
\[
\Bigl(\frac{\ln\rho}{\ln(1-\rho)}\Bigr)^{j} \kappa_\rho^{(j)} \sim \frac{(j-1)!\zeta(j)}{-\ln^{j}(1-\rho)}\ra 0
\]
as $\rho\uparrow1$. Since every term in \eqref{mmoments} is a multiple of some cumulant other than $\kappa_\rho^{(1)}$, except for the $(\kappa_\rho^{(1)})^m$ term, multiplying $E(X_\rho^m)$ by $\bigl(\frac{\ln\rho}{\ln(1-\rho)}\bigr)^m$ and taking the limit as $\rho\uparrow1$ eliminates all terms except for the afore mentioned term, where the limit converges to 1. Therefore, we can obtain the following asymptotic result:

\begin{corollary}\label{amoments}
For all $m\in\nats$, $E(X_\rho^m)\sim \bigl(\frac{\ln (1-\rho)}{\ln\rho}\bigr)^m$ as $\rho\uparrow1$.
\end{corollary}
To further clarify the special role of the number $Z_\rho=\frac{\ln (1-\rho)}{\ln\rho}$ in our model consider
\begin{eqnarray*}
U_\rho:=\mbox{number of frogs who reached the site $-Z_\rho$},
\end{eqnarray*}
where, for simplicity and clarity of the subsequent computation, we treat $Z_\rho$ as an integer number. Then, using \eqref{returns} and the Markov property,
observe that
\begin{eqnarray*}
E(U_\rho)=\sum_{x=0}^\infty P(S_n=-Z_\rho~\mbox{for some}~n\geq 1\,|\,S_0=x)=\sum_{x=0}^\infty \rho^{x+Z_\rho}=\frac{\rho^{Z_\rho}}{1-\rho}=1.
\end{eqnarray*}
This result can be heuristically interpreted as an illustration of the fact that $-Z_\rho$ serves as the most distant place, though barely,
is still accessible to the frog population (in average only one frog can reach that far).
\par
The next corollary is merely a couple of special cases of Theorem~\ref{moments}, but will be useful in the discussion of the asymptotics of $X_\rho$.

\begin{corollary}
\label{stats} For any $\rho\in(0,1)$, the expected value of $X_\rho$ is given by:
\begin{eqnarray}\label{expectation}
E(X_\rho)=\sum^\infty_{x=1} \frac{\rho^x}{1-\rho^x}=\frac{\psi_\rho(1)+\ln(1-\rho)}{\ln\rho},
\end{eqnarray}
where $\psi_\rho$ is the $q$-digamma function defined by \eqref{qpolygamma}. Also, the variance of $X_\rho$ is:
\begin{eqnarray}\label{variance}
\Var(X_\rho)=\sum_{x=1}^\infty \frac{x\rho^x}{1-\rho^x}=\frac{\psi'_\rho(1)}{\ln^2 \rho}.
\end{eqnarray}
\end{corollary}

We now seek to determine the asymptotic of the distribution of $X_\rho$ when $\rho\uparrow 1.$ Theorem~\ref{mode} and the result in \eqref{expectation} both hint that $X_\rho$ grows at roughly the same rate as $\frac{\ln(1-\rho)}{\ln\rho}$ as $\rho$ rises close to 1. The nature of this growth can be revealed by the asymptotic of the variance, which grows sufficiently slow to guarantee scaling limits for $X_\rho$.

\begin{theorem}
\label{pconvergence}
Let $Y_\rho$ share the same distribution as $\frac{\ln \rho}{\ln(1-\rho)}X_\rho$ for each $\rho\in (0,1)$. Then, as $\rho\uparrow1$, $Y_\rho \ra 1$ in probability. That is, for all $\epsilon > 0$, $$\lim_{\rho\ra 1^-}P(|Y_\rho-1|>\epsilon)=0.$$
\end{theorem}
\begin{proof}
By \eqref{expectation} and \eqref{asymp1}, $E(Y_\rho)=\frac{\ln\rho}{\ln(1-\rho)}E(X_\rho)\ra1$. To show the convergence in probability, it is enough to prove that $\Var(Y_\rho)\ra 0$ as $\rho\uparrow1$. By \eqref{variance} and \eqref{asymp2}, $\Var(X_\rho)\sim \frac{1!\zeta(2)}{(1-\rho)^2} \sim\frac{\pi^2}6\frac1{\ln^2\rho}$. Thus, $\Var(Y_\rho)=\frac{\ln^2\rho}{\ln^2(1-\rho)}\Var(X_\rho)\sim\frac{\pi^2}6\frac1{\ln^2(1-\rho)}\ra 0$. Since $E(Y_\rho)\ra1$ and $\Var(Y_\rho)\ra0$, convergence in probability follows.
\end{proof}
Aside from the probabilistic implications of Theorem \ref{pconvergence}, we can also use the generating functions of $Y_\rho$ to construct some limit identities involving the $q$-Pochhammer symbol. Since $Y_\rho\ra 1$ in probability (and thus in distribution), the moment generating function of $Y_\rho$ converges pointwise to that of the degenerate variable at 1. Similarly, the characteristic function also converges pointwise to the characteristic function of the same degenerate variable. This observation leads to the following corollary:
\begin{corollary}
For any $z>0$, $$\lim_{\rho\uparrow1}\dfrac{(\rho;\rho)_\infty}{\bigl(z^{\frac{\ln(\rho)}{\ln(1-\rho)}}\rho;\rho\bigr)_\infty} = z.$$
Also, for any $t\in\reals$, $$\lim_{\rho\uparrow1}\dfrac{(\rho;\rho)_\infty}{\bigl(e^{it\frac{\ln(\rho)}{\ln(1-\rho)}}\rho;\rho\bigr)_\infty} = e^{it}.$$
\end{corollary}
We wish to prove a stronger convergence of the sequence of random variables $\{Y_\rho\}_{\rho\in(0,1)}.$
Theorem~\ref{pconvergence} implies that an appropriate discretization $\{Y_{\rho_n}\}_{n\in\bbN}$ of $\{Y_\rho\}_{\rho\in(0,1)}$ can be chosen to achieve the almost sure convergence. The following result identifies a class of sequences $\{\rho_n\}_{n\in\bbN}$ that ensures the almost sure convergence of the discrete
sequence $Y_{\rho_n}.$
\begin{proposition}
Let $\{\rho_n\}_{n\in\nats}\subseteq (0,1)$ be a sequence such that $\rho_n \uparrow 1$ and $\frac1{\ln(1-\rho_n)}\in\ell^2$. Let $\{Y_{\rho_n}\}_{n\in\nats}$ be a sequence of random variables in the probability space $(\Omega,\mathcal{F},P)$ such that $Y_{\rho_n}$ has the same distribution as $\frac{\ln \rho_n}{\ln(1-\rho_n)} X_{\rho_n}$ for each $n$. Then, as $n\ra \infty$, $Y_{\rho_n} \ra 1$ a.\,s.
\end{proposition}
\begin{proof}
Let $\epsilon>0$ be given. By the Borel-Cantelli Lemma, a sufficient condition for a.\,s. convergence is that $P(|Y_{\rho_n}-1|>\epsilon)\in\ell^1$. Choose $\rho\in(0,1)$ such that $|E(Y_\rho)-1|<\frac{\epsilon}2$. Then, using Chebyshev's Inequality,
\[\begin{split}
P(|Y_\rho-1|>\epsilon) & \leq P(|Y_\rho-E(Y_\rho)|+|E(Y_\rho)-1|>\epsilon)\\
                        & \leq P\left(|Y_\rho-E(Y_\rho)|>\frac{\epsilon}2 \text{ or } |E(Y_\rho)-1|>\frac{\epsilon}2\right)\\
                        & = P\left(|Y_\rho-E(Y_\rho)|>\frac{\epsilon}2\right) \leq \frac4{\epsilon^2} \Var(Y_\rho).
\end{split}
\]
As noted above, $\Var(Y_\rho)=\frac{\ln^2 \rho}{\ln^2(1-\rho)} \Var(X_\rho) \sim \frac{\pi^2}6 \frac1{\ln^2(1-\rho)}$. Replacing $\rho$ with the terms of $\{\rho_n\}_{n\in\nats}$ described in the theorem statement, we see that $P(|Y_{\rho_n}-1|>\epsilon)\in\ell^1$, proving our result.
\end{proof}
The class of sequences defined in the hypothesis of the proposition above includes those of the form $\rho_n=1-e^{-n^c}$, where $c> \frac12$ is constant. For further research, we wish to broaden the class of sequences that lead to the a.\,s. convergence. For instance, instead of assuming that the models corresponding to different values of $\rho$ are independent, one can consider a standard hierarchial coupling of the underlying random walks leading to the setting where $P(X_{\rho_1}\leq X_{\rho_2})=1$
if $\rho_1<\rho_2.$ In that case one can consider for example $\rho_n=1-n^{-c},$ $c>\frac{1}{2},$ and imitating Etemadi's proof of the law of large numbers (see, for instance, Section~2.4 in \cite{durrett}), namely first considering subsequences $k(n)=\lfloor \alpha^n\rfloor$ with an arbitrary $\alpha>1$ and then using the fact that  the ratio of $\frac{\ln \rho_{k(n)}}{\ln(1-\rho_{k(n)})}$ and $\frac{\ln \rho_{k(n+1)}}{\ln(1-\rho_{k(n+1)})}$ converges to $\alpha^c$ when $n\to\infty,$ prove the almost sure convergence of $Y_{\rho_n}$ for the sequence $\rho_n=1-n^{-c}$ by finally taking $\alpha$ to $1.$
\section{More general frog distributions $\eta$}
\label{mgeneral}
Now let's consider the frog model with drift in which its frog distribution $\eta = \{\eta_x\}_{x=0}^\infty$ is a sequence of natural numbers, i.e., $\eta_x\geq 1$ for all $x\geq 0$ and $\eta_x=0$ elsewhere.
Main results of this section are stated in Theorems~\ref{pconvergence2} and~\ref{mpro} below.
\par
According to Theorem 2.1 in \cite{gantert}, in order to have transience in the frog model with drift $\rho$, and hence an almost surely finite minimum of its range, we must assume that $\sum_{x=0}^\infty \eta_x \rho^x <\infty$. Since we will be dealing with a continuum of choices for $\rho$, it will be useful to designate all of the frog distributions that will guarantee transience in the frog model with drift. Hence, we define the following set of integer sequences:
\[
\mathbb{H}:=\Bigl\{\eta\in \nats^{\nats\cup\{0\}}: \sum_{x=0}^\infty \eta_x \rho^x <\infty \text{ for all }\rho\in(0,1)\Bigr\}.
\]
It's worth noting that $\mathbb{H}$ does contain unbounded elements, such as $\eta=\{1,2,3,4,\dots\}$. In fact, any integer sequence $\eta$ such that $\eta_x =o(\alpha^x)$ as $x\to\infty$ for any $\alpha>1$ is in $\mathbb{H}$.\
\par
Similarly to the single-frog case, we can construct a family of independent random variables $(X_{\rho,\eta})_{\rho\in(0,1),\eta \subset\mathbb{H}}$ on the probability space $(\Omega,\mathcal{F},P)$ such that for each $\rho\in(0,1)$ and $\eta\in\mathbb{H}$, $X_{\rho,\eta}$ shares the same distribution as the negative of the minimum of the frog model with drift $\rho$ and frog distribution $\eta$.
\par
By using similar ideas as in the single-frog case, we can find the distribution of $X_{\rho,\eta}$:
\begin{eqnarray} \label{cdf}
P(X_{\rho,\eta}\leq x) =\prod_{k=0}^\infty \bigl(1-\rho^{x+k+1}\bigr)^{\eta_k}.
\end{eqnarray}
For simplicity, we will define the integer sequence $\{\Delta_k\}_{k=0}^\infty$ by $\Delta_0=\eta_0$ and for all $k\geq 1$, $\Delta_k=\eta_k-\eta_{k-1}$. The density of $X_{\rho,\eta}$ is then
\begin{eqnarray}\label{density}\begin{split}
P(X_{\rho,\eta}=x) &= \prod_{k=0}^\infty \bigl(1-\rho^{x+k+1}\bigr)^{\eta_k}-\prod_{k=0}^\infty \bigl(1-\rho^{x+k}\bigr)^{\eta_k}\\
&=\prod_{k=0}^\infty \bigl(1-\rho^{x+k+1}\bigr)^{\eta_k}\Bigl(1-\prod_{k=0}^\infty(1-\rho^{x+k})^{\Delta_k}\Bigr).
\end{split}
\end{eqnarray}
One special case to consider is $\eta_x=n\in\nats$ for all $x\geq 0$. Then, \eqref{cdf} and \eqref{density} become
\[
P(X_{\rho,\eta}\leq x)=\prod_{k=0}^\infty (1-\rho^{x+k+1})^n=(\rho^{x+1};\rho)_\infty^n,
\]
\[
P(X_{\rho,\eta}=x)=\bigl(1-(1-\rho^x)^n\bigr)(\rho^{x+1};\rho)_\infty^n.
\]
With frog distributions $\eta$ that differ from the single-frog case, we could assume that the moments of $X_{\rho,\eta}$ grow at different rates than $\frac{\ln(1-\rho)}{\ln\rho}$ found in Corollary \ref{amoments}. However, by the theorem below, if $\eta$ grows at a ``slow enough'' rate, the moments of $X_{\rho,\eta}$ will behave asymptotically similar to those of the single-frog case.
\begin{theorem}
\label{pconvergence2}
For each $\rho\in(0,1)$, let $Z_\rho=\frac{\ln(1-\rho)}{\ln\rho}$. Suppose $\{\eta_k\}_{k=0}^\infty\in\mathbb{H}$ such that $\lim_{\rho\uparrow 1} (1-\rho)^{1+\delta} \sum_{k=0}^\infty \eta_k \rho^k = 0$ for all $\delta>0$. Then, for all $m,n\in\nats$, $E(X_{\rho,\eta}^m)\sim Z_\rho^m$ as $\rho\uparrow 1$.
\end{theorem}
The proof of the theorem is given below in this section, after a shirt discussion of the result.  Note that according to the theorem, any even frog distribution $\eta=\{n,n,\dots\}$, where $n\in\nats$, will produce the same asymptotic rate for the moments. Not only that, but there exist unbounded choices for $\eta$ that produce the same rate as well. One simple example of such a choice is $\eta=\{1,2,3,\dots\}$.
\par
A consequence of Theorem \ref{pconvergence2} is an analogue to Theorem \ref{pconvergence} in the single-frog case, which reveals that $\frac{\ln(1-\rho)}{\ln\rho}$ is also an appropriate scaling limit for $X_{\rho,\eta}$'s convergence in probability as $\rho\uparrow1$.
\begin{theorem}
\label{mpro}
Let $\eta\in\mathbb{H}$ satisfy the conditions in the hypotheses of Theorem \ref{pconvergence2}. Let $Y_{\rho,\eta}$ share the same distribution as $\frac{\ln \rho}{\ln(1-\rho)}X_{\rho,\eta}$ for each $\rho\in (0,1)$. Then, as $\rho\uparrow 1$, $Y_{\rho,\eta} \ra 1$ in probability.
\end{theorem}
To clarify the intuition behind this result it is instructive to consider the case of an even frog configuration $\eta_x=m$ for all $x\geq 0,$ where $m\in\bbN$ is a fixed integer, and observe that the corresponding model can be thought of as a composition of $m$ independent models with $\eta_x=1.$ In this case, Theorem~\ref{mpro} is a direct implication of the result in Theorem~\ref{pconvergence} following by a simple observation that since the random variable $Y_\rho$ is asymptotic to a constant, the same is true for its analogue $Y_{\rho,\eta}$ in Theorem~\ref{mpro} which is the minimum of $m$ independent copies of $Y_\rho.$ From this perspective, Theorems~\ref{pconvergence2} and Theorem~\ref{mpro} can be viewed as an indirect extension of this argument to sequences $\eta_x$ growing sufficiently slowly, so that
they can be well enough approximated by initial configurations with an even distribution of frogs (notice that the further is the initial placement of a frog from the origin the less relevant it is for the asymptotic of $Y_{\rho,\eta}$).
\par
Before we begin the proof of Theorem \ref{pconvergence2}, we must first introduce a lemma.
\begin{lemma} \label{zrho}
For each $\rho\in(0,1)$, let $Z_\rho=\frac{\ln(1-\rho)}{\ln\rho}$. Then, for all $\delta>0$ and $m\in\nats$,
\begin{eqnarray}
\label{ratio}
\sum_{x=1}^\infty \bigl(Z_\rho(1+\delta)+x\bigr)^m \rho^x \sim \frac{Z_\rho^m (1+\delta)^m}{1-\rho}
\end{eqnarray}
as $\rho\uparrow1$.
\end{lemma}
\begin{proof}[Proof of Lemma~\ref{zrho}]
For the quotient of the left- and right-hand sides in  \eqref{ratio} we have
\[\begin{split}
(1-\rho)\sum_{x=1}^\infty \Bigl(1+\dfrac{x}{Z_\rho(1+\delta)}\Bigr)^m\rho^x&=(1-\rho)\sum_{x=1}^\infty \rho^x \sum_{j=0}^m {m\choose j}\frac{x^j}{Z_\rho^j(1+\delta)^j}\\
&=(1-\rho)\sum_{j=0}^m \frac1{Z_\rho^j(1+\delta)^j}{m\choose j} \Bigl(\sum_{x=1}^\infty x^j\rho^x\Bigr)\\
&\leq (1-\rho)\sum_{j=0}^m \frac1{Z_\rho^j(1+\delta)^j}{m\choose j}\frac{j!}{(1-\rho)^{j+1}}\\
&= \sum_{j=0}^m \frac{m!}{(m-j)!} \frac1{[Z_\rho(1+\delta)(1-\rho)]^j}.
\end{split}\]

Now, as $\rho\uparrow 1$, $\frac{\ln(\rho)}{1-\rho}\ra -1$ and $\ln(1-\rho)\ra -\infty$. Thus, $Z_\rho(1-\rho)\ra \infty$. Hence, for any $j>0$, the $j^\text{th}$ term of the above sum goes to 0 as $\rho\uparrow 1$. So when taking the limit, the only term in the sum that survives would be the $0^\text{th}$, which is equal to 1 for all $\rho$.
\end{proof}
We now proceed with the proof of the above theorem.
\begin{proof}[Proof of Theorem \ref{pconvergence2}]
Since the frog model corresponding to $X_{\rho,\eta}$ contains more frogs than the single-frog case, by Corollary \ref{amoments},
\[
\liminf_{\rho\uparrow1} \frac{E(X^m_{\rho,\eta})}{Z_\rho^m}\geq 1.
\]
 For the other inequality, note that for any $\eta\in\mathbb{H}$,
 \[
 E(X^m_{\rho,\eta})\leq E(X^m_{\rho,\theta}),
 \]
 where $\theta_k=\max\{\eta_j:j=1,2,\dots,k\}$. Thus, we can assume WLOG that $\eta$ is a nondecreasing sequence, and hence $\Delta_k\geq 0$ for all $k$.
\par
 Choosing $\delta>0$, we consider the sum for $E(X^m_{\rho,\eta})$ and split it at the point $\lfloor Z_\rho(1+\delta)\rfloor$. For the tail sum, we find that
\[\begin{split}
\sum_{x=\lfloor Z_\rho(1+\delta)\rfloor+1}^\infty x^m P(X_{\rho,\eta}=x)&\leq \sum_{x=\lfloor Z_\rho(1+\delta)+1\rfloor}^\infty x^m \Bigl(1-\prod_{k=0}^\infty (1-\rho^{x+k})^{\Delta_k}\Bigr)\\
& \sim \sum_{x=1}^\infty \bigl(Z_\rho (1+\delta)+x\bigr)^m \Bigl(1-\prod_{k=0}^\infty \bigl(1-\rho^{Z_\rho (1+\delta)}\rho^{x+k}\bigr)^{\Delta_k}\Bigr).
\end{split}\]
Since $\rho^{Z_\rho}=(1-\rho)$, we can see, after expanding the infinite product up to the first-order terms, that
\[\begin{split}
\sum_{x=\lfloor Z_\rho(1+\delta)\rfloor+1}^\infty x^m P(X_{\rho,\eta}=x)& \leq \sum_{x=1}^\infty \bigl(Z_\rho (1+\delta)+x\bigr)^m \Bigl(1-\prod_{k=0}^\infty \bigl(1-(1-\rho)^{1+\delta}\rho^{x+k}\bigr)^{\Delta_k}\Bigr)\\
& \leq \sum_{x=1}^\infty \bigl(Z_\rho (1+\delta)+x\bigr)^m \Bigl(\sum_{k=0}^\infty \Delta_k (1-\rho)^{1+\delta} \rho^{x+k}\Bigr)\\
& = (1-\rho)^{1+\delta} \Bigl( \sum_{k=0}^\infty \Delta_k \rho^k \Bigr)\sum_{x=1}^\infty \bigl(Z_\rho (1+\delta)+x\bigr)^m \rho^x\\
& \sim (1-\rho)^{1+\delta} \Bigl( \sum_{k=0}^\infty \eta_k \rho^k \Bigr)Z_\rho^m (1+\delta)^m.
\end{split}\]
The equivalence result in last line comes from Lemma \ref{zrho}, combined with the fact that $\sum_{k=0}^\infty \Delta_k \rho^k=\eta_0 + (1-\rho)\sum_{k=1}^\infty \eta_k \rho^k$.
\par
We also derive the following upper bound for the finite sum:
\[\begin{split}
\sum_{x=0}^{\lfloor Z_\rho (1+\delta)\rfloor} x^m \,P(X_{\rho,\eta}=x)&\leq \lfloor Z_\rho\rfloor^m (1+\delta)^m P\bigl(X_{\rho,\eta}\leq Z_\rho(1+\delta)\bigr)\\
& \leq Z_\rho^m (1+\delta)^m.
\end{split}\]
Thus, we can derive an upper bound for the following limit:
\[\begin{split}
\limsup_{\rho\uparrow1} \frac{E(X^m_{\rho,\eta})}{Z_\rho^m}&\leq \lim_{\rho\uparrow 1} (1+\delta)^m \Bigl(1+(1-\rho)^{1+\delta}\cdot\sum_{k=0}^\infty \eta_k \rho^k \Bigr)\\
&=(1+\delta)^m.
\end{split}\]
Since $\delta>0$ is arbitrary, $\limsup_{\rho\uparrow1} \frac{E(X^m_{\rho,\eta})}{Z_\rho^m}\leq 1$, and this completes the proof.
\end{proof}
\section{Initial configuration $\eta$ supported on the whole $\ints$}
\label{zconf}
In this section we will provide asymptotic bounds for the minimum of the frog model's range under the assumption that $\eta$ is supported on all of $\ints$. The main result of this section is stated in Theorem~\ref{bounds} below.
\par
Up until now, we have assumed that there were no sleeping frogs on any of the negative sites. With this assumption, all of the frogs on $\ints_+$ would eventually wake w.p.1, and we only needed to observe the collective minima of those frogs. However, when we consider the transient frog model with configuration $\eta$ supported on all of $\ints$, we now have a random number of active frogs originating from the negative sites that have the potential of expanding the range. We begin to explore the moments of the minimum of this case by the groundwork laid in the previous sections for the $\eta$ supported only on nonnegative sites.
\par
Fix any $n\in\mathbb N$. We will assume throughout this section that $\eta_x=n$ for any $x\in\ints,$ that is exactly $n$ frogs are initially placed at each site of $\ints.$ Our proofs in this section rely on the following description of the ``avalanche structure" of the model. We refer to the frogs initially located in the nonnegative sites of $\ints$ as the ``first wave". If we just observe the nonnegative frogs, we can locate the left-most site visited by the frogs from the first wave. We consider the frogs on the negative sites down to the minimum ever visited by the first wave to be the ``second wave'' of frogs being activated. Tracking the left-most site visited by the frogs from the second wave, we designate a ``third wave" of frogs activated between the subsequent minimums. We will continue to label these activated frogs in terms of waves. Since we assume a transient model, there will eventually be a final wave of frogs w.p.1 that never venture any more to the left than their initial locations. In this section, let the negative of the left-most site visited by any of the active frogs be $\witi X_{\rho,n}$.
\par
The following theorem provides upper and lower bounds for the $m^{\text{th}}$ moment of $\witi X_{\rho,n}$ for any given $m\in\nats$. While the bounds contain the familiar $Z_\rho=\frac{\ln(1-\rho)}{\ln(\rho)}$ term from Sections~\ref{single} and~\ref{mgeneral}, they are not immediately obvious from the previous results.
\begin{theorem}
\label{bounds}
The following holds for any $m\in\nats:$
\begin{itemize}
\item[(a)] The $m^{\text{th}}$ moment of $\witi X_{\rho,n}$ is bounded above by a function $\phi:(0,1)\ra[0,\infty)$ such that
\beqn
\label{upper}
\phi(\rho)\sim Z_\rho^m\Bigl(\frac{1-\rho}{2\pi}\Bigr)^{\frac{n}2}\exp\Bigl\{\frac{\pi^2}{6}\frac{mn}{1-\rho}\Bigr\} \quad \mbox{as}\quad \rho\uparrow1.
\feqn
\item[(b)] For any function $\delta:(0,1)\ra(0,1)$ such that $\lim_{\rho\uparrow1}\delta(\rho)=0$ and $\lim_{\rho\uparrow1}(1-\rho)^{\delta(\rho)}=0$, the $m^{\text{th}}$ moment of $\witi X_{\rho,n}$ is bounded below by a function $\psi_\delta:(0,1)\ra[0,\infty)$ such that
    \beqn
    \label{lower}
    \psi_\delta(\rho)\sim m!\, Z_\rho^m\exp\Bigl\{\frac{mn}{(1-\rho)^{\delta(\rho)}}\Bigr\}\quad \mbox{as} \quad \rho\uparrow1.
    \feqn
\end{itemize}
\end{theorem}
\begin{remark}
An example of a function $\delta$ that satisfies the conditions of Theorem \ref{bounds} (b) is $\delta(\rho)=\bigl\{\ln \bigl(\frac1{1-\rho}\bigr)\bigr\}^{-\alpha}$ for a constant $\alpha\in(0,1)$ and $\rho>1-e^{-1}$ (as long as we are only interested in the asymptotic as $\rho\uparrow 1,$ the values of $\delta(\rho)$ can be assigned arbitrarily for $\rho\leq 1-e^{-1}$). In this case, $\psi_\delta(\rho)\sim m!\, Z_\rho^m\exp\bigl\{mn\exp\bigl(|\ln(1-\rho)|^{1-\alpha}\bigr)\bigr\}$ as $\rho\uparrow1.$ For the sake of comparison with the upper bound, notice that the latter can be written as $\phi(\rho)\sim Z_\rho^m\bigl(\frac{1-\rho}{2\pi}\bigr)^{\frac{n}2}\exp\bigl\{mn\exp\bigl(|\ln(1-\rho)|\bigr)\bigr\}$ as $\rho\uparrow1,$ and that the parameter $\alpha\in (0,1)$ can be chosen arbitrarily close to zero.
\end{remark}
To motivate and clarify the intuition behind the coupling construction employed in the proof of Theorem~\ref{bounds} given below, we precede the proof by
the following observation.
\begin{remark}
\label{r1}
Consider the model described in Section~\ref{single}, namely $\eta_x=1$ for $x\geq 0$ and $\eta_x=0$ for $x<0.$ Fix any $\delta>0$ and consider
the probability $P_{\rho,\delta}$ that no one of the frogs initially placed to the right of $Z_\rho(1+\delta)$ will ever reach zero. For simplicity and without loss of generality we will treat $Z_\rho(1+\delta)$ as an integer. Then, since $\rho^{Z_\rho(1+\delta)}=(1-\rho)^{1+\delta},$
\beq
P_{\rho,\delta}=\prod_{j=1}^\infty \bigl(1-\rho^{Z_\rho(1+\delta)+j}\bigr)=\prod_{j=1}^\infty \bigl(1-(1-\rho)^{1+\delta}\rho^j\bigr).
\feq
Since $1-x>e^{-2x}$ for all $x>0$ small enough, we obtain that for $\rho$ sufficiently close to $1,$
\beq
P_{\rho,\delta}\geq \exp\Bigl\{-2(1-\rho)^{1+\delta}\sum_{j=1}^\infty \rho^j\Bigr\}=\exp\bigl\{-2\rho(1-\rho)^\delta\bigr\},
\feq
and hence
\beqn
\label{observe}
\lim_{\rho \uparrow 1}P_{\rho,\delta}=1.
\feqn
Now, recall that the results of Section~\ref{single} indicate tight concentration of the distribution of $X_\rho$ around $Z_\rho$ as $\rho\uparrow 1.$
On the other hand, heuristically, \eqref{observe} indicates that for large values of $\rho$ only the first $Z_\rho$ frogs are relevant to the dynamics of the model.
To further support this claim, consider the probability $Q_{\rho,\delta}$ that no one of the $Z_\rho(1-\delta)$ frogs initially placed at the first
$Z_\rho(1-\delta)$ nonnegative integers will ever reach $-Z_\rho(1-\delta).$ Similarly as before, we treat $Z_\rho(1-\delta)$ as an integer. Then
\beq
Q_{\rho,\delta}&=&\prod_{j=0}^{Z_\rho(1-\delta)-1} \bigl(1-\rho^{Z_\rho(1-\delta)+j}\bigr)\leq
\exp\Bigl\{-(1-\rho)^{1-\delta}\sum_{j=0}^{Z_\rho(1-\delta)-1}\rho^j\Bigr\}
\\
&=&
\exp\Bigl\{-(1-\rho)^{-\delta}\bigl(1-\rho^{Z_\rho(1-\delta)}\bigr)\Bigr\}=\exp\Bigl\{-(1-\rho)^{-\delta}\bigl(1-(1-\rho)^{1-\delta}\bigr)\Bigr\}
\\
&\leq&
\exp\Bigl\{-(1-\rho)^{-\delta}\bigl(1-e^{-\rho(1-\delta)}\bigr)\Bigr\}.
\feq
Thus
\beqn
\label{observe1}
\lim_{\rho \uparrow 1}Q_{\rho,\delta}=0.
\feqn
Heuristically, \eqref{observe} along with \eqref{observe1} tell us that the dynamics of the model considered in Section~\ref{single} is
for large values of $\rho$ similar to the dynamics of a modification where $\eta_x$ equals $1$ only if $0\leq x<Z_\rho$ and is $0$ otherwise.
The proof of Theorem~\ref{bounds} given below is using an interpretation of the model considered in this section as an ``avalanche" of the models
described in Section~\ref{mgeneral} and is using the above heuristic observation to produce upper and lower bounds of Theorem~\ref{bounds} for the model range.
\end{remark}
We now turn to a formal proof of Theorem~\ref{bounds}.
\begin{proof}[Proof of Theorem~\ref{bounds}]
$\mbox{}$ \\
\textit{(a)} We first provide an upper bound for the moments of $\witi X_{\rho,n}$ by coupling the frog model with the following variant. First observe the minimum location reached by the frogs that are initially placed on the nonnegative sites and obtain a second wave of active frogs. Modify the original second wave in the following way. Put $n$ more frogs on each site to the right of the minimum reached by the first wave and suppose that only the second wave can activate them. Note that without the consideration of activation times, this set of frogs is a translation of the configuration of nonnegative frogs considered in Sections~\ref{single} and~\ref{mgeneral}. Find the minimum bound for this modified configuration and, for the resulting next wave, add frogs to all of the right-side sites in a similar fashion. Since there is a positive probability for the nonnegative frogs never reach $-1,$ and each wave of frogs is a translation of this case, the waves will terminate w.p.1. Let $W_{\rho,n}$ be the negative of the minimum bound for this model.
\par
A formal definition of the distribution of $W_{\rho,n}$ can be given in the following manner. Define the sequence $\witi \eta=\{\witi \eta_x\}_{x\in\ints}\in \ints_+^\ints$ as follows:
\beq
\witi \eta_x=
\left\{
\begin{array}{ll}
n&\mbox{\rm if}~x\geq 0\\
0&\mbox{otherwise}.
\end{array}
\right.
\feq
Such a configuration has been considered in Section~\ref{mgeneral}. Let $X_{\rho,n}$ be the negative of the range of the corresponding model and
let $X_{\rho,n}^{(k)},$ $k\in\ints,$ be independent copies of this random variable. Let
\beq
T_{\rho,n}=\inf\{k\in\nats: X_{\rho,n}^{(k)}=0\}
\feq
and
\beq
W_{\rho,n}=\sum_{k=1}^{T_{\rho,n}-1} X_{\rho,n}^{(k)},
\feq
where, as usual, the empty sum (when $T_{\rho,n}=1$) is interpreted as zero.
\par
To facilitate our computations we will actually use the following
equivalent modification of this definition. Let
\begin{eqnarray*}
\varepsilon_{\rho,n}=P(X_{\rho,n}=0)=\bigl((\rho;\rho)_\infty\bigr)^n,
\end{eqnarray*}
where $X_{\rho,n}$ is the minimum of the range of the case with $n$ frogs on each nonnegative site, introduced in Section~\ref{mgeneral}. Thus
$\varepsilon_\rho=P\bigl(X_{\rho,n}^{(k)}=0\bigr)$ for any $k\in\nats.$ Let $\witi T_{\rho,n}$ be a geometric random variable with parameter $\varepsilon_\rho.$ Namely,
\begin{eqnarray*}
P\bigl(\witi T_{\rho,n}=k\bigr)=(1-\varepsilon_\rho)^k\varepsilon_\rho, \qquad k=0,1,\ldots.
\end{eqnarray*}
Notice that, $\lim_{\rho\to 1}\varepsilon_{\rho,n}=0,$ and hence $\witi T_{\rho,n}$ converges to infinity in probability as $\rho\uparrow 1.$
Clearly, $\witi T_{\rho,n}$ has the same distribution as $T_{\rho,n}-1.$ Let ${\mathcal Y}_{\rho,n}=\{Y^{(k)}_{\rho,n}\}_{k\in\ints}$ be an i.i.d. sequence independent of $\witi T_{\rho,n}$ and such that for any $j\in\nats,$
\begin{eqnarray*}
P(Y_{\rho,n}=j)=P(X_{\rho,n}=j|X_{\rho,n}>0)=\frac{P(X_{\rho,n}=j)}{1-P(X_{\rho,n}=0)}.
\end{eqnarray*}
Finally, let
\begin{eqnarray*}
\witi W_{\rho,n}=\sum_{k=1}^{\witi T_{\rho,n}} Y^{(k)}_{\rho,n}.
\end{eqnarray*}
Clearly, $\witi W_{\rho,n}$ has the same distribution as $W_{\rho,n}.$
\par
For $m\in\nats$, we look at the $m^\text{th}$ moment of $\witi W_{\rho,n}$ conditioned on $\witi T_{\rho,n}$:
\beq
E\bigl(\witi W_{\rho,n}^m\bigl|\witi T_{\rho,n}\bigr)&=&
E\Bigl[\Bigl({\sum}_{k=1}^{\witi T_{\rho,n}}\, Y^{(k)}_{\rho,n}\Bigr)^m \Big| \witi T_{\rho,n} \Bigr]\leq \Bigl({\sum}_{k=1}^{\witi T_{\rho,n}}\, E\bigl[\bigl(Y_{\rho,n}^{(k)}\bigr)^m\bigr]^{1/m}\Bigr)^m
\\
&=&
\bigl(\witi T_{\rho,n}\bigr)^m E\bigl[\bigl(Y_{\rho,n}^{(1)}\bigr)^m\bigr],
\feq
where we use Minkowski inequality and the fact that $Y_{\rho,n}^{(k)}$ are independent of $\witi T_{\rho,n}.$ From this conditioned expectation, we approximate the $m^\text{th}$ moments of $W_{\rho,n}$ as $\rho\uparrow 1$:
\beqn
\label{wald}
E\bigl(W_{\rho,n}^m\bigr)=E\bigl(E\bigl(\witi W_{\rho,n}^m\bigl|\witi T_{\rho,n}\bigl)\bigl)\leq E\bigl[\bigl(\witi T_{\rho,n}\bigr)^m\bigr]\cdot E\bigl[\bigl(Y_{\rho,n}^{(1)}\bigr)^m\bigr]\sim \varepsilon_{\rho,n}^{-m} Z_\rho^m,
\feqn
where $Z_\rho=\frac{\ln(1-\rho)}{\ln\rho}$ as in Section~\ref{mgeneral}. To evaluate the moments of $\witi T_{\rho,n}$ we used the following known result whose
short proof is supplied for reader's convenience.
\begin{lemma}
\label{mg}
For $\veps\in (0,1),$ let $T_\veps$ be a geometric random variable with parameter $\varepsilon_\rho.$ Namely,
\begin{eqnarray*}
P\bigl(\witi T_\veps=k\bigr)=(1-\varepsilon)^k\varepsilon, \qquad k=0,1,\ldots.
\end{eqnarray*}
Then, for any $m\in\nats,$ $E(T_\veps^m)\sim \veps^{-m}$ as $\veps\uparrow 1.$
\end{lemma}
\begin{proof}[Proof of Lemma~\ref{mg}]
We have as $\veps\uparrow 1,$
\begin{eqnarray*}\label{geometric}
E(T_\veps^m)&=&\sum^\infty_{k=0} P(T_\veps^m>k)=\int_0^\infty P(T_\veps^m>x)\,dx=\int_0^\infty P(T_\veps>y)\cdot my^{m-1}\,dy\\
&\sim& \int_0^\infty (1-\veps)^{y+1}\cdot my^{m-1}\,dy=m(1-\veps)\int_0^\infty e^{\ln(1-\veps)y}y^{m-1}\,dy\\
&=&m(1-\veps)\frac{\Gamma(m)}{|\ln(1-\veps)|^m}\sim \frac{m!}{\veps^m}.
\end{eqnarray*}
The proof of the lemma is complete.
\end{proof}
Using \eqref{wald}, we finally arrive at the asymptotic bound in \textit{(a)} through the asymptotic of the $q$-Pochhammer symbol derived from \cite{qasymp}: for $\rho=e^{-t}$, as $t\downarrow0$,
\[
(\rho;\rho)_\infty \sim \sqrt{\frac{2\pi}{t}} \exp\Bigl(\frac{-\pi^2}{6t}\Bigr).
\]
Note that $t=-\ln\rho\sim 1-\rho$, and we have \textit{(a)}. \qed
\\
$\mbox{}$
\\
\textit{(b)} Now, consider another variant of the frog model. Define a function $\delta(\rho):(0,1)\to (0,1).$ To simplify notation, we will occasionally use $\delta_\rho$ for $\delta(\rho).$ For a given $\rho$, consider the configuration $\hat \eta$ with $\hat \eta_k=n$ if $k\in\{0,1,2,\dots, Z_\rho(1-\delta_\rho)\}$ and $\hat \eta_k=0$ elsewhere. For simplicity and without loss of generality, we will assume that $Z_\rho(1-\delta_\rho)$ is integer-valued. Start the model, and see if the frogs eventually reach the site $-Z_\rho(1-\delta_\rho)$. If they do, activate $n$ frogs on each of the $ Z_\rho(1-\delta_\rho)$ sites to the left of the origin. Now observe if the newly activated frogs reach the site $-2 Z_\rho(1-\delta_\rho)$. If they do, activate $n$ frogs on all the $ Z_\rho(1-\delta_\rho)$ sites to the left of the sites previously activated. Continue this procedure indefinitely. Note that at any observed time step, this model will always have fewer active frogs than the frog model with initial configuration $\eta$ with $\eta_k=n$ for all $k\in\ints.$ By Theorem 2.1 in \cite{gantert}, the frog model with this configuration is transient. Hence, the variant model is transient, and the process will eventually stop producing new sets of frogs from the left. Let $V_{\rho,\delta}$ be the negative of the minimum of the range of the variant, and let $\tau_{\rho,\delta}$ be the number of activated blocks of the length $Z_\rho(1-\delta_\rho),$ not including the initial one located within $\ints_+.$
\par
Let
\begin{eqnarray*}
\theta_{\rho,\delta}=\prod_{j=0}^{ Z_\rho(1-\delta_\rho)-1} \bigl(1-\rho^{Z_\rho(1-\delta_\rho)}\rho^j\bigr)^n=\bigl((\rho^{ Z_\rho(1-\delta_\rho)};\rho)_{Z_\rho(1-\delta_\rho)}\bigr)^n.
\end{eqnarray*}
By this definition, $\theta_{\rho,\delta}$ is the probability that none of the frogs located at the first $Z_\rho(1-\delta_\rho)$
nonnegative sites will ever reach the half-line on the left of $Z_\rho(1-\delta_\rho).$ The event that a given block of frogs $nZ_\rho(1-\delta_\rho)$ active frogs reachs its goal and produces another active frogs block of length $Z_\rho(1-\delta_\rho)$ is independent of that of any other set. Hence, viewing the event of a set reaching its destination $Z_\rho(1-\delta_\rho)$ to the left as a ``failure'', we have that $\tau_{\rho,\delta}$ is geometrically distributed with parameter $\theta_{\rho,\delta}.$ Namely,
\begin{eqnarray*}
P(\tau_{\rho,\delta}=k)=(1-\theta_{\rho,\delta})^k\theta_{\rho,\delta}, \qquad k=0,1,\ldots.
\end{eqnarray*}

Notice that since $\rho^{Z_\rho}=1-\rho,$
\begin{eqnarray*}
\theta_{\rho,\delta}&=&\prod_{j=0}^{ Z_\rho(1-\delta_\rho)-1} \bigl(1-(1-\rho)^{1-\delta_\rho}\rho^j\bigr)^n\leq \exp\Bigl(-n(1-\rho)^{1-\delta_\rho}\sum_{j=0}^{Z_\rho(1-\delta_\rho)-1}\rho^j\Bigr)
\\
&=&
\exp\Bigl(-n\frac{1-\rho^{Z_\rho(1-\delta_\rho)}}{(1-\rho)^{\delta_\rho}}\Bigr)=\exp\Bigl(-n\frac{1-(1-\rho)^{1-\delta_\rho}}{(1-\rho)^{\delta_\rho}}\Bigr).
\end{eqnarray*}
With the $\delta$ chosen with the constraints specified in the theorem's hypotheses, we see that $\theta_{\rho,\delta}\ra 0$ as $\rho\uparrow1$. Hence, by Lemma~\ref{mg},
$E(\tau_{\rho,\delta}^m)\sim m!\,\theta_{\rho,\delta}^{-m}$ as $\rho\uparrow 1.$ Clearly, $\witi X_{\rho,n}$ is stochastically dominated from below
by $V_{\rho,\delta}=\tau_{\rho,\delta} \cdot Z_\rho(1-\delta_\rho).$  The lower bound for the moments of $\witi X_{\rho,n}$ is therefore
\begin{eqnarray*}
E(\tau_{\rho,\delta}^m)\cdot Z_\rho^m(1-\delta_\rho)^m&\sim&\frac{m!}{\theta_\rho^m}\cdot Z_\rho^m
\sim
m!\,Z_\rho^m\exp\Bigl\{ mn\frac{1-(1-\rho)^{1-\delta_\rho}}{(1-\rho)^{\delta_\rho}}\Bigr\}
\\
&\sim&
 m!\,Z_\rho^m\exp\Bigl\{\frac{mn}{(1-\rho)^{\delta_\rho}}\Bigr\},
\end{eqnarray*}
as $\rho\uparrow 1.$ The proof of the theorem is complete.
\end{proof}


\end{document}